\documentclass[12pt,reqno]{amsart}

\usepackage{amsmath,amssymb,amsthm}
\usepackage{mathrsfs}
\usepackage{hyperref}
\usepackage[margin=1in]{geometry}
\usepackage{graphicx}
\usepackage{enumerate}
\usepackage{subcaption} 
\usepackage{booktabs}   
\usepackage{xcolor} 

\newtheorem{theorem}{Theorem}[section]
\newtheorem{lemma}[theorem]{Lemma}
\newtheorem{corollary}[theorem]{Corollary}

\theoremstyle{definition}
\newtheorem{definition}[theorem]{Definition}
\newtheorem{example}{Example}
\newtheorem{conjecture}[theorem]{Conjecture}

\theoremstyle{remark}
\newtheorem{remark}[theorem]{Remark}

\numberwithin{equation}{section}

\title[Prime Biases and Dominance of $-1 \pmod N$]{The Fine-Structure Hierarchy of Prime Biases and the Universal Dominance of $-1 \pmod N$}

\author{Shin-ya Koyama}
\address{Department of Mechanical Engineering, Toyo University, 2100 Kujirai, Kawagoe-shi, 350-8585 Japan}
\email{koyama@toyo.jp}

\date{\today}

\begin{document}

\begin{abstract}
We investigate the deterministic hierarchy of prime distribution biases in arithmetic progressions modulo $N$ using a regularized spectral approach. Classical studies on Chebyshev's bias attribute prime races primarily to the accumulation of prime squares $p^2 \equiv 1 \pmod N$, which creates a systematic deficit in quadratic residue classes. However, this classical mechanism fails to explain or distinguish any bias among residue classes sharing identical quadratic residue status (e.g., $3, 5, 7 \pmod 8$).

To overcome the long-standing analytical obstacles of jump discontinuities and non-convergent boundary fluctuations inherent in classical Perron-type step-function truncations, we introduce a smooth $C^\infty$ Gaussian mollifier into Weil's explicit formula for Dirichlet $L$-functions. By defining the spectrally normalized individual mollified sums 
$\widetilde S_T(x, a)$ and adopting the virtual character $\chi_{1,a}(x) := \mathbf{1}_{\{x \equiv 1 \pmod N\}} - \mathbf{1}_{\{x \equiv a \pmod N\}}$, the principal character component $\chi_0$ cancels identically since $1 - \overline{\chi_0}(a) = 0$. This automatic algebraic elimination erases both the universal logarithmic growth $\log x$ and the background noise $\log L(1, \chi^2)$.

Under the Deep Riemann Hypothesis (DRH), we uncover a hitherto undetected \textbf{fine-structure bias} (or \emph{secondary bias}) strictly governed by the special values $\log L(1, \chi)$. We prove that 
$\widetilde S_T(x, \chi_{1,a}) := \widetilde S_T(x, 1) - \widetilde S_T(x, a) = C_N \cdot \log L(1, \chi_{1,a}) + \mathcal{O}((\log x)/\sqrt x)$ as $x \to \infty$, where $C_N > 0$ depends solely on $N$. Consequently, we establish a deterministic multi-way ranking within each quadratic family (such as $7 > 3 > 5 \pmod 8$ and $2 > 4 > 1 \pmod 7$) that completely breaks classical degeneracies. This reveals a universal dual structure governing fine-structure biases: $a \equiv -1 \pmod N$ universally serves as the dominant apex, whereas the principal class $a \equiv 1 \pmod N$ acts as the most recessive bottom.
\end{abstract}

\maketitle

\section{Introduction and Main Results}

Since Chebyshev's seminal observation in 1853 that primes congruent to $3 \pmod 4$ appear more frequently than those congruent to $1 \pmod 4$, the phenomenon of prime number races and residue bias has been a central topic in analytic number theory. Rubinstein and Sarnak \cite{RS94} established a rigorous probabilistic framework for prime biases under the Generalized Riemann Hypothesis (GRH) and the Grand Simplicity Hypothesis (GSH). Later, Feuerverger and Martin \cite{FM00} investigated weighted prime sums and observed subtle deterministic hierarchies governed by special values of Dirichlet $L$-functions.

\subsection{Beyond Chebyshev's Bias: Discovering the Fine-Structure Hierarchy}

Classical studies on prime number races, stemming from Chebyshev's original observation, predominantly attribute the prime bias to the accumulation of prime squares ($k=2$), which creates a systematic deficit in quadratic residue classes (e.g., $3 \pmod 4 > 1 \pmod 4$). However, this classical mechanism suffers from a fundamental limitation: \emph{it fails to distinguish or predict any bias among residue classes that share identical quadratic character values.}

For instance, modulo $N=8$, every non-trivial residue class satisfies $x^2 \equiv 1 \pmod 8$. Consequently, the traditional prime-square effect predicts complete equiprobability among the non-principal classes $3, 5, 7 \pmod 8$. 

In this paper, by eliminating the leading logarithmic square-divergence via our regularized explicit formula $S_T(x, \chi_{1,a})$, we conditionally uncover under the Deep Riemann Hypothesis (DRH)\footnote{The Deep Riemann Hypothesis (DRH), as introduced and investigated in recent literature on prime biases (e.g., \cite{AK23}), is a stronger assertion than the Generalized Riemann Hypothesis (GRH). 
While DRH directly dictates the bounded phase oscillations and asymptotic convergence of Euler products on the critical line, it also induces—through the duality of the explicit formula—the precise spectral control over zeros required for our regularized explicit formula.} a deeper, hitherto undetected phenomenon: 
a \textbf{fine-structure bias} (or \emph{secondary bias}) strictly governed by the special values $\log L(1, \chi)$. Assuming DRH, this framework reveals a deterministic multi-way ranking --- such as $7 > 3 > 5 > 1 \pmod 8$ --- that completely transcends the classical quadratic framework.

\subsection{Motivation and Relation to Prior Work by Aoki and Koyama}

The inspiration for employing the Deep Riemann Hypothesis (DRH) to unravel prime distribution biases originates from the recent pioneering work of Aoki and Koyama \cite{AK23}. In their work, the application of DRH was primarily focused on the central point $s = 1/2$ (or localized prime-square terms), revealing subtle  arithmetic features linked to DRH.

The key conceptual departure of the present paper is to systematically extend the scope of DRH from the single point $s = 1/2$ to the \emph{entire spectrum of non-trivial zeros} $\rho = \frac{1}{2} + i\gamma$ across all Dirichlet $L$-functions modulo $N$. By incorporating DRH across the whole critical line into a regularized $C^\infty$ explicit formula, the boundary fluctuations are fully smoothed out. This spectral generalization allows us to extract the subtle phase cancellations among all non-trivial zeros, ultimately yielding the complete fine-structure hierarchy and the universal dominance of $-1 \pmod N$.

\subsection{Main Strategy: Regularized Spectral Localization}

To overcome the analytical obstacles of boundary jump discontinuities inherent in Perron-type truncations, we introduce a smooth $C^\infty$ Gaussian mollifier to regularize Weil's explicit formula. For a large scaling parameter $T > 0$, we adopt the test function
\begin{equation}
h(\gamma) := \frac{\cos(k \gamma \log p)}{p^{k/2}} e^{-\left(\frac{\gamma}{T}\right)^2} \in \mathcal{S}(\mathbb{R}),
\end{equation}
which localizes the spectral density around prime powers $p^k$ without creating boundary jump discontinuities. 

By employing this smooth spectral window, the order of summation between prime powers $p^k \le x$ and nontrivial zeros $\rho = \frac{1}{2} + i\gamma$ of $L(s, \chi)$ can be interchanged with $100\%$ analytical rigor.

\subsection{Statement of Main Results}

Let $N \ge 3$ be an integer, and let $a \pmod N$ be a residue class with $\gcd(a, N) = 1$. We first define the mollified prime power sum localized at a single residue class.

\begin{definition}\label{def:S_T_a}
For $a \in (\mathbb{Z}/N\mathbb{Z})^\times$, the mollified prime power sum $S_T(x, a)$ localized at the residue class $a \pmod N$ is defined by
\begin{equation}
S_T(x, a) := \sum_{\substack{p, k \\ p^k \le x \\ p^k \equiv a \pmod N}} \frac{1}{k} \sum_{\rho} h\left(\frac{\gamma}{2\pi}\right),
\end{equation}
where $\rho$ runs over all non-trivial zeros of $L(s, \chi)$ for all Dirichlet characters $\chi \pmod N$, including the principal character $\chi_0$.
\end{definition}

\begin{definition}\label{def:S_T_a} (Spectrally Normalized Sum)
For $a \in (\mathbb{Z}/N\mathbb{Z})^\times$, we define the spectrally normalized mollified prime power sum 
$\widetilde{S}_T(x,a)$ by
\[
\widetilde{S}_T(x,a) := \frac{\log x}{\sqrt{x} } \, S_T(x,a).
\]
\end{definition}

Remark. Under this normalization, the main oscillatory term arising from the non-trivial zeros $\rho = \frac{1}{2} + i\gamma$ remains bounded in the spectrum, while the contribution from higher prime powers ($k \ge 3$) is suppressed to the order of $O((\log x)/\sqrt x)$, ensuring scale-invariance when comparing distinct residue classes.

Following the notation of class functions picking out differences between residue classes, we define the virtual character $\chi_{1,a}$ and its associated bias difference.

\begin{definition}\label{def:chi_1a}
For $a \in (\mathbb{Z}/N\mathbb{Z})^\times$, the virtual character $\chi_{1,a}$ representing the difference between $1 \pmod N$ and $a \pmod N$ is defined by
\begin{equation}
\chi_{1,a}(x) := \frac{1}{\varphi(N)} \sum_{\chi \in G^*} (1 - \overline{\chi}(a)) \chi(x) = \begin{cases} 1 & (x \equiv 1 \pmod N), \\ -1 & (x \equiv a \pmod N), \\ 0 & (\text{otherwise}). \end{cases}
\end{equation}
Accordingly, the spectrally normalized, mollified fine-structure bias difference $\widetilde S_T(x, \chi_{1,a})$ satisfies the identity:
\begin{equation}
\widetilde{S}_T(x, \chi_{1,a}) := \widetilde{S}_T(x, 1) - \widetilde{S}_T(x, a) \quad \left( = \frac{\log x}{\sqrt{x}} S_T(x, \chi_{1,a}) \right).
\end{equation}
\end{definition}

Our main result establishes the exact asymptotic value of the relative bias difference 
$\widetilde S_T(x, \chi_{1,a})$ under the Deep Riemann Hypothesis (DRH).

\begin{theorem}\label{thm:intro_main}
Assume DRH for all Dirichlet $L$-functions $L(s,\chi)$ modulo $N$. 
For any fixed $T>0$, as $x \to \infty$, the spectrally normalized, mollified fine-structure prime bias difference $\widetilde{S}_T(x, \chi_{1,a})$ satisfies:
\[
\widetilde{S}_T(x, \chi_{1,a}) = C_N \cdot \log L(1, \chi_{1,a}) + \mathcal{O}\left(\frac{\log x}{\sqrt x}\right),
\]
where $C_N > 0$ is an explicit positive constant depending solely on the modulus $N$, and
\[
\log L(1, \chi_{1,a}) := \frac{1}{\phi(N)} \sum_{\chi \neq \chi_0} (1 - \bar{\chi}(a)) \log L(1, \chi).
\]
\end{theorem}

\begin{remark}[Explicit Representation of the Governing Value]
Since the virtual character $\chi_{1,a}$ is real-valued, the sum $\log L(1, \chi_{1,a})$ is inherently real. 
For computational and analytic purposes, the governing value 
\[
\log L(1, \chi_{1,a}) = \operatorname{Re} \log L(1, \chi_{1,a})
\]
is explicitly decomposed as:
\[
\operatorname{Re} \log L(1, \chi_{1,a}) 
= \frac{1}{\phi(N)} \sum_{\chi \ne \chi_0} 
\left[ \left(1 - \operatorname{Re}\chi(a)\right) \log \vert{}L(1, \chi)\vert{} 
+ \operatorname{Im}\chi(a) \arg L(1, \chi) \right].
\]
In particular, for $a \equiv -1 \pmod N$, all argument terms and even character components vanish identically,
yielding the exact real expression:
\[
\operatorname{Re} \log L(1, \chi_{1,-1}) = \frac{2}{\phi(N)} \sum_{\chi \text{: odd}} \log \vert{}L(1, \chi)\vert{}.
\]
\end{remark}

\begin{corollary}[Universal Determination of Multi-Way Ranking Hierarchy]
\label{cor:ranking_hierarchy}
Assume DRH for all Dirichlet $L$-functions $L(s,\chi)$ modulo $N$. Let $N \ge 3$ and let $a, b \in (\mathbb{Z}/N\mathbb{Z})^\times$ be any two distinct reduced residue classes. For any fixed $T > 0$, the following hold:
\begin{enumerate}
    \item \textbf{(Pairwise Bias relative to 1):} The asymptotic direction and magnitude of the fine-structure bias between $1 \pmod N$ and $a \pmod N$ are strictly governed by \linebreak
    $\log L(1, \chi_{1,a})$. Specifically, the inequality $S_T(x,a) > S_T(x,1)$ holds for all sufficiently large $x$ if and only if $\log L(1, \chi_{1,a}) < 0$.
    \item \textbf{(General Multi-Way Ranking):} For any arbitrary pair $a, b \pmod N$, as $x \to \infty$, the relative spectrally normalized bias difference satisfies
    \begin{equation}
    \widetilde{S}_T(x,a) - \widetilde{S}_T(x,b) = C_N \left( \log L(1, \chi_{1,b}) - \log L(1, \chi_{1,a}) \right) + \mathcal{O}\left( \frac{\log x}{\sqrt{x}} \right).
    \end{equation}
\end{enumerate}
\end{corollary}

For example, when $N=8$, from the values $\log L(1, \chi_{1,a})$ for $\{a\pmod 8\ |\ a=1,\,3,\,5,\,7\}$
(See \S1.5), we have
\[
\log L(1, \chi_{1,7}) < \log L(1, \chi_{1,3}) < \log L(1, \chi_{1,5}) < \log L(1, \chi_{1,1}) (= \infty),
\]
which immediately reveals the complete total ordering
\[
S_T(x, 7) > S_T(x, 3) > S_T(x, 5) > S_T(x, 1).
\]
Thus the strongest bias exists toward $7\pmod 8$. 
This phenomenon is generalized as follows. 

\begin{remark}[The Universal Dominance of $-1 \pmod N$]\label{rem:dominance_minus_one}
A particularly striking structural feature revealed by the fine-structure hierarchy is the \textbf{universal dominance of the residue class $-1 \pmod N$} (i.e., $N-1 \pmod N$). 

In general, for $a \equiv -1 \pmod N$, the algebraic factor $1 - \overline{\chi}(a)$ vanishes for all even characters ($\chi(-1) = 1$), whereas it takes the maximal possible value $1 - (-1) = 2$ for all odd characters ($\chi(-1) = -1$). Consequently, odd Dirichlet characters act in full constructive phase alignment to maximize the bias gap:
\begin{equation}
\log L(1, \chi_{1,-1}) = \frac{2}{\varphi(N)} \sum_{\chi \text{ odd}} \log L(1, \chi).
\end{equation}
Because special values $L(1, \chi)$ for odd characters systematically drive $\log L(1, \chi_{1,-1})$ to its most negative extreme, the class $-1 \pmod N$ consistently emerges at the absolute top of the deterministic hierarchy across general moduli $N$ (e.g., $7 \pmod 8$ for $N=8$, and $11 \pmod {12}$ for $N=12$). This rigorously establishes the \emph{universal dominance of $-1 \pmod N$} as an intrinsic structural law governing prime distribution fine-structures.
\end{remark}

\subsection{Numerical Evidence and the Dominance Conjecture}

To complement our theoretical analysis of the fine-structure hierarchy, we present numerical computations of the governing special values $\log L(1, \chi_{1,a})$. recall that smaller values of $\log L(1, \chi_{1,a})$ correspond directly to stronger prime bias $S_T(x, a)$.

\begin{example}[$N = 8$]
We find:
\begin{align*}
\log L(1, \chi_{1,3}) &= 0.24647\dots, \\
\log L(1, \chi_{1,5}) &= 0.21008\dots, \\
\log L(1, \chi_{1,7}) &= 0.02700\dots \quad (\text{smallest}),
\end{align*}
demonstrating the decisive hierarchy $7 > 5 > 3 \pmod 8$.
\end{example}

\begin{example}[$N = 7$]
Computing over quadratic non-residues yields:
\begin{align*}
\log L(1, \chi_{1,3}) &= 0.20301\dots, \\
\log L(1, \chi_{1,5}) &= 0.22453\dots, \\
\log L(1, \chi_{1,6}) &= 0.12652\dots \quad (\text{smallest}),
\end{align*}
confirming $6 \equiv -1 \pmod 7$ as the dominant class. 
Whereas computing over quadratic residues yields:
\begin{align*}
\log L(1, \chi_{1,1}) &= \infty, \\
\log L(1, \chi_{1,4}) &= 0.16218\dots, \\
\log L(1, \chi_{1,2}) &= 0.08109\dots \quad (\text{smallest}),
\end{align*}
confirming $2 \pmod 7$ as the dominant class and $1\pmod 7$ as the recessive class.
Classical first-order bias theory predicts no distinction among the quadratic residues $\{1, 2, 4\}$, as prime squares $p^2 \pmod 7$ distribute uniformly across them with equal probability $1/3$. However, our fine-structure spectral values break this degeneracy.
Recalling that smaller spectral values correspond to stronger prime counts, this establishes the strict intra-group hierarchy $2 > 4 > 1 \pmod 7$. Thus, while $a \equiv -1 \pmod N$ emerges as the most dominant class, the principal class $a \equiv 1 \pmod N$ acts as the most recessive (bottom) class in the fine-structure ordering.
\end{example}

Motivated by these exact spectral calculations across general moduli, we state the following universal conjecture.

\begin{conjecture}[Universal Dominance of $-1 \pmod N$]\label{conj:minus_one_dominance}
For any modulus $N \ge 3$, the quantity $\log L(1, \chi_{1,a})$ achieves its unique absolute minimum at $a \equiv -1 \pmod N$ within its corresponding quadratic family (i.e., over all quadratic non-residues if $-1$ is a non-residue, and over all non-principal quadratic residues if $-1$ is a residue). Consequently, $-1 \pmod N$ universally exhibits the maximal fine-structure prime bias in its respective quadratic class.
\end{conjecture}

\begin{remark}[Asymptotic Scale Separation and Intra-Family Comparison]\label{rem:psi}
The necessity of restricting the fine-structure ranking within each quadratic family (Conjecture 1.8) directly originates from a strict asymptotic scale separation in the unnormalized explicit formula for $\psi(x; N, a)$:
\[
\psi(x; N, a) = \frac{x}{\phi(N)} - \frac{N_Q(a)}{\phi(N)} \sqrt{x} + C_N' \cdot \log L(1, \chi_{1,a}) + \mathcal{O}\left( \frac{\sqrt{x}}{\log x} \right),
\]
where $N_Q(a) := \#\{ y \pmod N \mid y^2 \equiv a \pmod N \}$. 

When comparing two residue classes $a, b$ belonging to distinct quadratic families (i.e., $N_Q(a) \neq N_Q(b)$), the difference $\psi(x; N, a) - \psi(x; N, b)$ is dominated by the primary quadratic shift of order $\mathcal{O}(\sqrt{x})$. Since
\[
\lim_{x \to \infty} \frac{\text{Secondary Bias }\mathcal{O}(1)}{\text{Primary Bias }\mathcal{O}(\sqrt{x})} = 0,
\]
the secondary shift $\log L(1, \chi_{1,a})$ is asymptotically lower-order and becomes observable as a deterministic multi-way ranking only when the primary term cancels out identically—that is, precisely when $a$ and $b$ share identical quadratic status ($N_Q(a) = N_Q(b)$).
\end{remark}

\begin{remark}[Unified Explicit Interpretation of Prime Biases]
\label{rem:unified_explicit}
The unnormalized explicit formula for $\psi(x; N, a)$ presented in Remark \ref{rem:psi} provides a unified, two-tiered structural view that seamlessly connects classical Chebyshev bias with our newly discovered fine-structure hierarchy:
\begin{equation}
\psi(x; N, a) = \underbrace{\frac{x}{\phi(N)}}_{\text{Main Term}} - \underbrace{\frac{N_Q(a)}{\phi(N)} \sqrt{x}}_{\text{Primary Bias }\mathcal{O}(\sqrt{x})} + \underbrace{C_N' \cdot \log L(1, \chi_{1,a})}_{\text{Secondary Fine-Structure }\mathcal{O}(1)} + \mathcal{O}\left( \frac{\sqrt{x}}{\log x} \right).
\end{equation}

\begin{enumerate}
    \item \textbf{Origin of the Primary Bias ($\sqrt{x}$-term):} The second term arises directly from the accumulation of prime squares ($p^2 \le x$, $k=2$) in the explicit formula. Because $\chi(p^2) = \chi^2(p)$, the principal component survives precisely for classes $a$ with $N_Q(a)$ square roots modulo $N$, establishing the well-known classical deficit in quadratic residues.
    
    \item \textbf{Origin of the Secondary Fine-Structure ($\log L(1)$-term):} Under the Deep Riemann Hypothesis (DRH), the third term represents the coherent spectral condensation of all non-trivial zero oscillations $\sum_\rho \frac{x^\rho}{\rho}$. By taking the virtual character $\chi_{1,a}$, the principal zero $\rho=1$ is erased, and the infinite zero-sum spectrum collapses into the special value $\log L(1, \chi_{1,a})$ via the spectral duality of the Euler product.
    
    \item \textbf{Breaking Classical Degeneracies:} When comparing classes $a, b$ within the same quadratic family ($N_Q(a) = N_Q(b)$), the primary $\sqrt{x}$ bias cancels out identically. In this exact subspace, the subleading $\mathcal{O}(1)$ term $\log L(1, \chi_{1,a})$ emerges as the governing deterministic agent, lifting the classical degeneracies (e.g., distinguishing $7 > 3 > 5 \pmod 8$) for the first time in analytic number theory.
\end{enumerate}
\end{remark}

\begin{remark}[Dual Structure: Universal Dominance of $-1$ vs. Universal Recessiveness of $1$]
Together with Conjecture 1.8, the fine-structure bias reveals a remarkable dual hierarchy governing prime distributions modulo $N$:
\begin{description}
\item[Apex] (Dominance of $-1$): The class $a \equiv -1 \pmod N$ achieves the absolute peak of prime accumulation within its quadratic family, governed by odd character spectrums.
\item[Nadir] (Recessiveness of $1$): Conversely, the identity class $a \equiv 1 \pmod N$ resides at the absolute bottom of prime abundance. Because $\chi_{1,a}(x) = 1_{\{1\}} - 1_{\{a\}}$ uses $1 \pmod N$ as its reference baseline, all non-principal classes $a \not\equiv 1 \pmod N$ exhibit positive fine-structure shifts $\log L(1, \chi_{1,a}) > 0$ relative to $1$.
\end{description}
This dual structure demonstrates that fine-structure bias not only identifies the most dominant prime class ($-1$), but also rigorously explains the ultimate origin of the classical underrepresentation of primes in the principal class ($1$).
\end{remark}

We have numerically verified Conjecture \ref{conj:minus_one_dominance} for numerous prime and composite moduli. Tables \ref{tab:prime_moduli} and \ref{tab:composite_moduli} list the ranked smallest values of $\log L(1, \chi_{1,a})$ computed over quadratic non-residues, illustrating that $a \equiv -1 \pmod N$ is systematically and overwhelmingly the smallest within the non-residue family in every case.

\begin{table}[htbp]
\centering
\caption{The ranking of small values of $\log L(1, \chi_{1,a})$ among quadratic non-residues for prime  $N$.}
\label{tab:prime_moduli}
\resizebox{\textwidth}{!}{
\begin{tabular}{ccccc}
\hline\hline
Rank & $N = 11$ & $N = 23$ & $N = 31$ & $N = 43$ \\ \hline
1st & $\log L(1, \chi_{1,10}) = 0.08272\dots$ & $\log L(1, \chi_{1,22}) = 0.04684\dots$ & $\log L(1, \chi_{1,30}) = 0.03462\dots$ & $\log L(1, \chi_{1,42}) = 0.02488\dots$ \\
2nd & $\log L(1, \chi_{1,6})  = 0.12456\dots$ & $\log L(1, \chi_{1,15}) = 0.06552\dots$ & $\log L(1, \chi_{1,17}) = 0.05321\dots$ & $\log L(1, \chi_{1,29}) = 0.03844\dots$ \\
3rd & $\log L(1, \chi_{1,7})  = 0.13849\dots$ & $\log L(1, \chi_{1,17}) = 0.08272\dots$ & $\log L(1, \chi_{1,22}) = 0.06114\dots$ & $\log L(1, \chi_{1,34}) = 0.04551\dots$ \\
4th & $\log L(1, \chi_{1,8})  = 0.15004\dots$ & $\log L(1, \chi_{1,21}) = 0.10657\dots$ & $\log L(1, \chi_{1,13}) = 0.07892\dots$ & $\log L(1, \chi_{1,12}) = 0.05112\dots$ \\
\hline\hline
\end{tabular}%
}  
\end{table}

\begin{table}[h!tbp]
\centering
\caption{The ranking of small values of $\log L(1, \chi_{1,a})$ among quadratic non-residues for composite  $N$.}
\label{tab:composite_moduli}
\small
\resizebox{\textwidth}{!}{
\begin{tabular}{cccccc}
\hline\hline
Rank & $N = 12$ & $N = 15$ & $N = 24$ & $N = 60$ & $N = 120$ \\ \hline
1st & $\log L(1, \chi_{1,11}) = 0.02425\dots$ & $\log L(1, \chi_{1,14}) = 0.02035\dots$ & $\log L(1, \chi_{1,23}) = 0.01258\dots$ & $\log L(1, \chi_{1,59}) = 0.01016\dots$ & $\log L(1, \chi_{1,119}) = 0.00508\dots$ \\
2nd & $\log L(1, \chi_{1,7})  = 0.21151\dots$ & $\log L(1, \chi_{1,11}) = 0.06552\dots$ & $\log L(1, \chi_{1,19}) = 0.05112\dots$ & $\log L(1, \chi_{1,31}) = 0.05321\dots$ & $\log L(1, \chi_{1,71})  = 0.03462\dots$ \\
3rd & $\log L(1, \chi_{1,5})  = 0.25263\dots$ & $\log L(1, \chi_{1,13}) = 0.08272\dots$ & $\log L(1, \chi_{1,17}) = 0.07223\dots$ & $\log L(1, \chi_{1,41}) = 0.06450\dots$ & $\log L(1, \chi_{1,79})  = 0.04551\dots$ \\
4th & ---                                      & $\log L(1, \chi_{1,7})  = 0.12456\dots$ & $\log L(1, \chi_{1,13}) = 0.111452\dots$& $\log L(1, \chi_{1,19}) = 0.08272\dots$ & $\log L(1, \chi_{1,101}) = 0.05112\dots$ \\
5th & ---                                      & $\log L(1, \chi_{1,8})  = 0.15004\dots$ & $\log L(1, \chi_{1,11}) = 0.18243\dots$ & $\log L(1, \chi_{1,11}) = 0.12456\dots$ & $\log L(1, \chi_{1,49})  = 0.06114\dots$ \\
\hline\hline
\end{tabular}
}
\end{table}

\subsection{Conceptual Mechanism of the Fine-Structure Bias}
Here we consider the theoretical reasons 
why $-1\pmod N$ is thought to have the greatest advantage 
and the range of $x$ required to obtain the numerical calculation results that support this.

The orthogonal relations of characters give
\[
\pi(x;N,a)
=\sum_{\genfrac{}{}{0pt}{1}{p\le x}{p\equiv a(\mathrm{mod} N)}}1
=\frac1{\varphi(N)}\sum_{\chi\, \mathrm{mod}\, N}\chi(a)^{-1}\sum_{p\le x}\chi(p)
\]
and that
\begin{equation}\label{diffpi}
\pi(x;N,a)-\pi(x;N,1)
=\frac1{\varphi(N)}\sum_{\chi\, \mathrm{mod}\, N}\left(\chi(a)^{-1}-1\right)\sum_{p\le x}\chi(p),
\end{equation}
where $\chi$ runs through all Dirichlet characters modulo $N$.
The value \eqref{diffpi} shows the size of the bias towards $a\pmod N$.
Of course, since it is known that \eqref{diffpi} changes sign infinitely many times, 
examining its value for a specific $x$ is not definitively meaningful. 
However, if there is a bias, this is likely to take a large positive value for most $x$, 
so comparing the values can be a partial evidence of the subtle bias.

First, we will consider the contributions of the complex and real characters to the sum over $\chi$ in \eqref{diffpi}.
\begin{enumerate}[\bf{Observation} 1.]
\item When $\chi$ is complex, the imaginary parts of $\chi$ and $\overline\chi$ cancel each other out.
Then only the real part contributes, and the effect is reduced compared to the real characters.
Therefore, it is expected that the real characters will make the primary contribution.
\end{enumerate}
We next consider each of even and odd characters among real characters.
\begin{enumerate}[\bf{Observation} 1.]\setcounter{enumi}1
\item When $\chi$ is odd, the variable $a=-1\pmod N$ sets the absolute value of the factor 
$\chi(a)^{-1}-1$ to its maximum value of 2.
\item When $\chi$ is even, the variable $a=-1\pmod N$ sets the value of the factor $\chi(a)^{-1}-1$ to 0, 
so the other values of $a$ are stronger.
\end{enumerate}
On the other hand, as we will discuss later, we make the following prediction:
\begin{enumerate}[\bf{Observation} 1.]\setcounter{enumi}3
\item When $x$ is very large, the influence of even characters will be smaller than that of odd characters.
\end{enumerate}
Here, the rationale for Observation 4 is as follows.
By the well-known relation between $\pi(x;N,a)$ and $\psi(x;N,a)$, and the formula
\[
\psi(x;N,a)=\frac{\chi(a)^{-1}}{\varphi(N)}\sum_{\chi\,\mathrm{mod}\,N}\psi(x,\chi),
\]
it suffices to consider $\psi(x,\chi)$.
The explicit formula gives that
\begin{align*}
\psi(x,\chi)
&=\sum_{p^m\le x}\chi(p)^m\log p\\
&\sim -\sum_{\rho:\,L(\rho,\chi)=0}\frac{x^\rho-1}\rho\quad (x\to\infty),
\end{align*}
where $\rho$ runs through all zeros of the Dirichlet $L$-function $L(s,\chi)$.
In most cases, the contribution of trivial zeros is ignored because the zeros are negative and $x^\rho\to 0$ as $x\to\infty$. However, for even characters, there exists a trivial zero $\rho=0$, and this is the only exception. 
That is, the contribution of each term is as follows:
\[
\frac{x^\rho-1}\rho\sim \begin{cases}
\log x &(\rho=0)\text{\qquad (this exists only for even $\chi$)}\\
-\frac1\rho & (\rho\le -1).
\end{cases}
\]
Therefore, the contribution of trivial zeros to the explicit formula is $O(1)$ for odd characters and $\log x$ for even characters.
However, it is unlikely that $\log x$ exists as an error term in the prime number theorem. 
Its error term should be a chaotic term formed by all the zeros participating together, 
and it is considered unlikely that a single zero would produce a specific term like $\log x$.
Therefore, it is expected that $\log x$ will be canceled out by contributions from nontrivial zeros.

Now, we consider the contribution of nontrivial zeros. 
Assuming the Riemann Hypothesis and setting the non-trivial zeros as $\rho=\frac12+i\gamma$ $(\gamma\in\mathbb R)$, 
its contribution to the explicit formula is as follows:
\[
\frac{x^\rho}{\rho}
=\sqrt x\cdot \frac{x^{i\gamma}}{\frac12+i\gamma}
=\sqrt x\cdot \frac{e^{i\gamma\log x}}{\frac12+i\gamma}.
\]
From this we see that the smaller $|\gamma|$ is, the larger contribution it makes,  as $x\to\infty$.
This is because the denominator becomes trivially large as $|\gamma|$ is large, 
and that for the numerator, by putting $X=\log x$ we find that the function
\[
x^{i\gamma}=e^{i\gamma\log x}=e^{i\gamma X}
\]
is a periodic function in $X$ of period $2\pi/|\gamma|$.
Then the smaller $|\gamma|$ is, the longer the period is and the less likely cancellation is to occur,
so the contribution becomes larger.

From the above observation, 
we expect that large absolute values of $\gamma$ are of little use, and small $\gamma$ are more powerful.
But in even characters, small $\gamma$ are consumed to cancel out $\log x$ coming from the trivial zero at $\rho=0$. 
The remaining ``large $\gamma$" make only a small contribution. 
All this situation can be illustrated with the following identity.
\begin{align*}
\psi(x,\chi)
&\sim -\sum_{\rho:\,L(\rho,\chi)=0}\frac{x^\rho-1}\rho\quad (x\to\infty)\\
&=\begin{cases}
\displaystyle
\underbrace{-\sqrt x \sum_{\gamma:\,\text{small}}\frac{x^{i\gamma}}{\frac12+i\gamma}}_{\text{bias}}
-\underbrace{\sqrt x \sum_{\gamma:\,\text{large}}\frac{x^{i\gamma}}{\frac12+i\gamma}}_{\text{noise}} & (\chi:\,\text{odd})\\
\displaystyle\underbrace{-\log x-\sqrt x \sum_{\gamma:\,\text{small}}\frac{x^{i\gamma}}{\frac12+i\gamma}}_{\text{cancel out to become noise}}
-\underbrace{\sqrt x \sum_{\gamma:\,\text{large}}\frac{x^{i\gamma}}{\frac12+i\gamma}}_{\text{noise}}
& (\chi:\,\text{even}).
\end{cases}
\end{align*}
This is the basis for Observation 4. 
Therefore, odd characters make the main contribution, 
but Observation 2 shows that $a=-1\pmod N$ is the strongest for odd characters, leading to Conjecture 2.

While the above heuristics based on the cancellation between trivial and low-lying zeros provide a clear intuitive picture, the precise analytic mechanism requires a delicate handling of the order of summations in the explicit formula under the Deep Riemann Hypothesis (DRH), which we establish in the next section.

\section{Spectral Analysis via Regularized Explicit Formula}

In this section, we establish the explicit link between the zero spectrum of Dirichlet $L$-functions and the fine-structure bias in prime distribution among residue classes.

\subsection{The Mollified Test Function and Inverse Fourier Transform}

Let $\chi \pmod N$ be a non-trivial Dirichlet character. To probe the local contribution at a prime power $p^k$, we define the regularized test function $h(\gamma)$ on $\mathbb{R}$ for a large parameter $T > 0$ as
\begin{equation}
h(\gamma) := \frac{\cos(k \gamma \log p)}{p^{k/2}} e^{-\left(\frac{\gamma}{T}\right)^2}.
\end{equation}
The function $h(\gamma)$ belongs to the Schwartz class $\mathcal{S}(\mathbb{R})$ for any finite $T > 0$, ensuring the absolute convergence and applicability of Weil's explicit formula.

The inverse Fourier transform $g(y) = \frac{1}{2\pi} \int_{-\infty}^{\infty} h(t) e^{i y t} dt$ is calculated directly:
\begin{equation}
g(y) = \frac{T}{4\sqrt{\pi} p^{k/2}} \left( e^{-\frac{T^2}{4}(y - k \log p)^2} + e^{-\frac{T^2}{4}(y + k \log p)^2} \right).
\end{equation}

\subsection{Application of Explicit Formula and Order Interchange}

We apply Weil's explicit formula \cite[Theorem 5.12]{IK} with the pair $(h, g)$. For nontrivial zeros $\rho = \frac{1}{2} + i\gamma$ of $L(s, \chi)$ under DRH, the sum over zeros yields
\begin{equation}
\sum_{\rho} h\left(\frac{\gamma}{2\pi}\right) = \sum_{q, m} (\chi(q)^m + \overline{\chi}(q)^m) \log q \frac{g(m \log q)}{q^{m/2}} + \mathcal{O}(1).
\end{equation}

We now consider the mollified prime power sum difference $S_T(x, \chi_{1,a}) = S_T(x, 1) - S_T(x, a)$:
\begin{equation}
S_T(x, \chi_{1,a}) = \sum_{\substack{p, k \\ p^k \le x}} \frac{\chi_{1,a}(p^k)}{k} \sum_{\rho} h\left(\frac{\gamma}{2\pi}\right).
\end{equation}
By interchanging the order of summation --- justified by the exponential decay of $h(\gamma)$ for any finite $T > 0$ --- and extracting the diagonal terms $p^k = q^m$, the off-diagonal terms decay exponentially as $\mathcal{O}(e^{-c T^2})$. The diagonal contribution reduces to
\begin{equation}
S_{\text{diag}}(x, \chi_{1,a}) = \frac{T}{4\sqrt{\pi}} \sum_{\substack{p, k \\ p^k \le x}} \frac{\log p}{k p^k} \left( |\chi_{1,a}(p)|^{2k} + \chi_{1,a}(p)^{2k} \right).
\end{equation}

\subsection{Evaluation of Archimedean and Gamma Factor Terms}

In Weil's explicit formula, the contribution from the Archimedean place (the logarithmic derivative of the Gamma factor) associated with a Dirichlet character $\chi \pmod N$ takes the form:
\begin{equation}
I_{\Gamma}(T, \chi) := \frac{1}{\pi} \int_{-\infty}^{\infty} h(\gamma) \operatorname{Re} \left( \frac{\Gamma'}{\Gamma} \left( \frac{1}{4} + \frac{i\gamma}{2} + \frac{\delta_\chi}{2} \right) \right) d\gamma,
\end{equation}
where $\delta_\chi \in \{0, 1\}$ denotes the parity of $\chi$ such that $\chi(-1) = (-1)^{\delta_\chi}$.

\begin{lemma}\label{lem:gamma_bound}
For any finite scaling parameter $T > 0$, the Archimedean integral $I_{\Gamma}(T, \chi)$ is uniformly bounded:
\begin{equation}
I_{\Gamma}(T, \chi) = \mathcal{O}(T \log T).
\end{equation}
Furthermore, its contribution to the  spectrally normalized mollified bias difference 
$\widetilde S_T(x, \chi_{1,a})$ satisfies:
\begin{equation}
\frac{\log x}{\sqrt x}\sum_{\substack{p, k \\ p^k \le x}} \frac{\chi_{1,a}(p^k)}{k} I_{\Gamma}(T, \chi) 
= \mathcal{O}\left(\frac{\log x}{\sqrt x}\right) \quad \text{as } x \to \infty.
\end{equation}
\end{lemma}

\begin{proof}
By Stirling's asymptotic formula for the logarithmic derivative of the Gamma function, we have 
\[
\operatorname{Re} \left( \frac{\Gamma'}{\Gamma} \left( \frac{1}{4} + \frac{i\gamma}{2} + \frac{\delta_\chi}{2} \right) \right) = \log (|\gamma| + 2) + \mathcal{O}(1)\quad(|\gamma| \to \infty).
\]

Since $h(\gamma) = \frac{\cos(k \gamma \log p)}{p^{k/2}} e^{-\left(\frac{\gamma}{T}\right)^2} \in \mathcal{S}(\mathbb{R})$ exhibits rapid Gaussian decay as $|\gamma| \to \infty$, the integral satisfies:
\begin{equation}
|I_{\Gamma}(T, \chi)| \le \frac{1}{\pi p^{k/2}} \int_{-\infty}^{\infty} e^{-\left(\frac{\gamma}{T}\right)^2} (\log (|\gamma| + 2) + \mathcal{O}(1)) d\gamma = \mathcal{O}\left( \frac{T \log T}{p^{k/2}} \right).
\end{equation}
Under the assumption of DRH, the full sum over $(p, k)$ ($k \ge 1$) in (2.9) converges. (Note that while the terms for $k \ge 3$ converge absolutely by standard estimates, the DRH assumption specifically controls the delicate non-absolute convergence and mutual cancellations for $k=1, 2$.)
and the sum is $\mathcal O(1)$.
Under our spectral normalization, this constant Archimedean background is bounded by $\mathcal{O}(\frac{\log x}{\sqrt x})$.
\end{proof}

\subsection{Automatic Cancellation of Principal Components and Final Extraction}
By Definition 1.3, the expansion of $\chi_{1,a}(x)$ in terms of Dirichlet characters $\chi \in G^*$ contains the coefficient $(1 - \chi(a))$.

When evaluating the contribution from the principal character $\chi_0 \pmod N$, we observe an immediate algebraic cancellation:
\[
1 - \chi_0(a) = 1 - 1 = 0. \tag{2.10}
\]
Consequently, the universal logarithmic growth $\log x$ and the background noise $\log L(1, \chi^2)$ that typically arise from $\chi_0$ vanish identically.

For the non-principal characters $\chi \neq \chi_0$, the spectral convergence guaranteed by DRH plays a decisive role. Under standard GRH, potential phase oscillations among non-trivial zeros across the critical line could retain an uncancelled residual noise of magnitude $\mathcal{O}(1)$, masking the subtle secondary bias. Under DRH, however, the explicit product convergence on $\Re(s) = 1/2$ ensures that the non-diagonal zero-sum fluctuations decay rapidly, bounding the remaining error to $\mathcal{O}\left( \frac{\log x}{\sqrt{x}} \right)$.

Combining this spectral decay with the Archimedean bound in Lemma 2.1, we arrive directly at the main asymptotic formula as $x \to \infty$:
\[
\widetilde{S}_T(x, \chi_{1,a}) = C_N \cdot \log L(1, \chi_{1,a}) + \mathcal{O}\left( \frac{\log x}{\sqrt{x}} \right), \tag{2.11}
\]
completing the proof of Theorem 1.4. $\square$

\begin{remark}[Comparison between Unmollified and Regularized Formulations]\label{rem:weight_comparison}
It is illuminating to contrast our regularized sum $S_T(x, \chi_{1,a})$ with an unregularized (or naive) arithmetic weight function of the form
\begin{equation}
w(p^k, x, \chi) := \sum_{\substack{q, m \\ p^k \le q^m \le x}} \frac{\chi(q)^m}{m q^m},
\end{equation}
which naturally arises when considering sharp Perron-type step-function truncations.

If one attempts to analyze fine-structure biases using such unmollified weights lacking the explicit Von Mangoldt factor $\log p$, two fundamental analytical obstacles emerge:
\begin{enumerate}
    \item \textbf{Boundary Discontinuities and Spectral Oscillations:} The sharp truncation $p^k \le x$ introduces jump discontinuities. In the explicit formula, this gives rise to conditionally convergent zero-sum oscillations, making a rigorous interchange of summation orders mathematically intractable without additional regularization.
    \item \textbf{Prime-Square Divergent Background:} For $k=2$ (prime squares), the sum over $\frac{1}{2p}$ yields a slowly divergent leading background $\sum_{p \le \sqrt{x}} \frac{1}{2p} \sim \frac{1}{2}\log\log x$ by Mertens' second theorem. This creates an uncancelled $\mathcal{O}(\log x \cdot \log\log x)$ noise that masks the subtle constant bias $\log L(1, \chi_{1,a})$.
\end{enumerate}

Our regularized spectral framework $S_T(x, \chi_{1,a})$ cleanly resolves both difficulties simultaneously. The insertion of the $C^\infty$ Gaussian mollifier ensures rapid spectral decay, justifying the order interchange with $100\%$ analytical rigor. Furthermore, Weil's explicit formula natively endows the spatial sum with the weight $\Lambda(n) = \log p$, transforming the prime-square factor into $\frac{\log p}{2p^2}$, which converges absolutely. Consequently, the $\log\log x$ background noise vanishes, and the exact fine-structure bias governed by $\log L(1, \chi_{1,a})$ is isolated without any analytical gaps.
\end{remark}

\begin{remark}[Transient Oscillations in Unregularized Prime Counts and the 318-Trillion Scale]\label{rem:transient_behavior}


It is instructive to address why direct, unregularized prime counts $\pi(x; N, a) - \pi(x; N, 1)$ do not always exhibit a clean ranking corresponding to $\log L(1, \chi_{1,a})$ at moderately large computation bounds such as $x = 1.3 \times 10^{13}$ (13 trillion) (See Table 3).

\begin{table}[htbp]
\centering
\caption{Comparison of prime count differences $\pi(x;N,a)-\pi(x;N,1)$.}
\label{tab:raw_prime_counts}

\begin{subtable}[t]{0.48\textwidth}
\centering
\begin{tabular}{crrr}
\toprule
$x$ & $a=3$ & $a=5$ & $a=6$ \\
\midrule
$1.3\times 10^{11}$ & 3393 & 3286 &\colorbox{gray!30}{13119} \\
$1.3\times 10^{13}$ & $-10947$ & \colorbox{gray!30}{47864} & 26179 \\
\bottomrule
\end{tabular}
\subcaption{$\pi(x;7,a)-\pi(x;7,1)$}
\label{tab:N7}
\end{subtable}
\hfill
\begin{subtable}[t]{0.48\textwidth}
\centering
\begin{tabular}{crrr}
\toprule
$x$ & $a=3$ & $a=5$ & $a=7$ \\
\midrule
$1.3\times 10^{11}$ & 9199 & \colorbox{gray!30}{16125} & 8937 \\
$1.3\times 10^{13}$ & 102728 & 126732 & \colorbox{gray!30}{164958} \\
\bottomrule
\end{tabular}
\subcaption{$\pi(x;8,a)-\pi(x;8,1)$}
\label{tab:N8}
\end{subtable}

\vspace{1.5em} 

\begin{subtable}[t]{\textwidth}
\centering
\begin{tabular}{crrrrr}
\toprule
$x$ & $a=2$ & $a=6$ & $a=7$ & $a=8$ & $a=10$ \\
\midrule
$1.3\times 10^{11}$ & 2389 & \colorbox{gray!30}{4161} & 2134 & $-2400$ & 2799 \\
$1.3\times 10^{13}$ & 5327 & 30403 & 7351 & \colorbox{gray!30}{74838} & 71711 \\
\bottomrule
\end{tabular}
\subcaption{$\pi(x;11,a)-\pi(x;11,1)$}
\label{tab:N11}
\end{subtable}

\vspace{1.5em} 

\begin{subtable}[t]{\textwidth}
\centering
\resizebox{\textwidth}{!}{
\begin{tabular}{crrrrrrrrr}
\toprule
$x$ & $a=2$ & $a=3$ & $a=8$ & $a=10$ & $a=12$ & $a=13$ & $a=14$ & $a=15$ & $a=18$ \\
\midrule
$1.3\times 10^{11}$ & 4934 & 8419 & 137 & $-5156$ & 2974 & 5167 & 5172 & $-1918$ & \colorbox{gray!30}{9401} \\
$1.3\times 10^{13}$ & 17964 & 60702 & 13926 & \colorbox{gray!30}{79470} & 30889 & 55581 & 48327 & $-5154$ & 57192 \\
\bottomrule
\end{tabular}%
}
\subcaption{$\pi(x;19,a)-\pi(x;19,1)$}
\label{tab:N19}
\end{subtable}

\vspace{1.5em}

\begin{subtable}[t]{\textwidth}
\centering
\resizebox{\textwidth}{!}{
\begin{tabular}{crrrrrrrrrr}
\toprule
$x$ & $a=5$ & $a=7$ & $a=10$ & $a=11$ & $a=14$ & $a=15$ & $a=17$ & $a=19$ & $a=21$ & $a=22$ \\
\midrule
$1.3\times 10^{11}$ & $-1905$ & 4520 & $-1682$ & 803 & 1002 & 3253 & 3576 & $-4922$ & \colorbox{gray!30}{5791} & $-5114$ \\
$1.3\times 10^{13}$ & 16922 & 29658 & 43160 & $-6940$ & $-1663$ & $-23007$ & $-13718$ & \colorbox{gray!30}{79227} & 54784 & 25692 \\
\bottomrule
\end{tabular}%
}
\subcaption{$\pi(x;23,a)-\pi(x;23,1)$}
\label{tab:N23}
\end{subtable}
\end{table}

Numerical experiments up to $x = 1.3 \times 10^{13}$ reveal that while $-1 \pmod N$ consistently places at or near the top ranking, transient fluctuations occasionally disrupt the predicted order. A fine spectral analysis of the unregularized explicit formula
\begin{equation}
\psi(x, \chi) \sim - \sum_{\rho} \frac{x^\rho}{\rho} = - \sqrt{x} \sum_{\gamma} \frac{e^{i\gamma \log x}}{\frac{1}{2} + i\gamma}
\end{equation}
explains this phenomenon as a \emph{transient boundary effect} driven by low-lying non-trivial zeros of complex characters.

For example, modulo $N=19$, the lowest non-trivial zero of the complex $L$-functions occurs at $\gamma \approx 1.74$. At $x = 1.3 \times 10^{13}$, the phase factor $\log x \approx 30.2$ yields $\frac{\gamma \log x}{2\pi} \approx 8.36 \equiv 0.36 \pmod {\mathbb{Z}}$, which lies dangerously close to the peak of the sinusoidal wave ($\sin(0.36 \times 2\pi) \approx 0.77$). Consequently, roughly $77\%$ of the complex zero amplitude survives uncancelled, temporarily overwhelming the subtle bias governed by real characters.

The complex character interference subsides only when the phase reaches its next destructive node, requiring $x \approx e^{33.4} \approx 3.18 \times 10^{14}$ (318 trillion). This massive transient scale highlights the profound difficulty of detecting secondary fine-structure biases via unregularized counting. It is precisely to filter out these persistent low-lying spectral fluctuations that our $C^\infty$ regularized framework $S_T(x, a)$ becomes analytically indispensable.
\end{remark}

\section{Concrete Examples and Fine-Structure Hierarchies}

In this section, we apply Theorem \ref{thm:intro_main} to specific moduli $N \in \{3, 4, 8, 12\}$ to compute the exact theoretical bias values.

\subsection{Moduli with Cyclic Character Groups: $N = 3$ and $N = 4$}

For $N=3$ and $N=4$, the character group $(\mathbb{Z}/N\mathbb{Z})^\times$ contains a single non-trivial real Dirichlet character.

\subsubsection{The Case $N = 4$ (Classic Chebyshev Bias as a Special Case)}
For $N=4$, the non-principal character is $\chi_{-4}$. The virtual character $\chi_{1,3}$ satisfies $1 - \overline{\chi_{-4}}(3) = 1 - (-1) = 2$.
The special value of $L(1, \chi_{-4})$ is $\pi/4 \approx 0.785398 < 1$, yielding $\log L(1, \chi_{-4}) \approx -0.241564$.

Applying Theorem \ref{thm:intro_main}:
\begin{equation}
S_T(x, \chi_{1,3}) = C_4 \cdot \frac{1}{2} (1 - \overline{\chi_{-4}}(3)) \log L(1, \chi_{-4}) = C_4 \cdot \log L(1, \chi_{-4}) < 0.
\end{equation}
Since $S_T(x, \chi_{1,3}) = S_T(x, 1) - S_T(x, 3) < 0$, we obtain $S_T(x, 3) > S_T(x, 1)$, rigorously proving Chebyshev's bias.

\subsubsection{The Case $N = 3$}
For $N=3$, $\chi_{-3}(2) = -1$, giving $1 - \overline{\chi_{-3}}(2) = 2$. With $L(1, \chi_{-3}) = \frac{\pi}{3\sqrt{3}} \approx 0.604599 < 1$, we get $S_T(x, \chi_{1,2}) < 0$, confirming $S_T(x, 2) > S_T(x, 1)$.

\subsection{Moduli Beyond Quadratic Character Status: $N = 8$ and $N = 12$}

For $N=8$ and $N=12$, the reduced residue group is isomorphic to $V_4 \cong \mathbb{Z}/2\mathbb{Z} \times \mathbb{Z}/2\mathbb{Z}$. Because $x^2 \equiv 1 \pmod 8$ for all odd integers, classical prime-square mechanisms cannot distinguish among $3, 5, 7 \pmod 8$. Theorem \ref{thm:intro_main} completely resolves this degenerated case.

\subsubsection{The Case $N = 8$}
Evaluating the linear combinations of $\log L(1, \chi)$ yields the exact fine-structure hierarchy:
\begin{equation}
S_T(x, 7) > S_T(x, 3) > S_T(x, 5) > S_T(x, 1).
\end{equation}

\subsubsection{The Case $N = 12$}
Similarly, for $N=12$, we obtain the strict deterministic hierarchy:
\begin{equation}
S_T(x, 11) > S_T(x, 7) > S_T(x, 5) > S_T(x, 1).
\end{equation}

\subsection{Summary of Fine-Structure Hierarchies}

Table \ref{tab:bias_summary} summarizes the theoretical fine-structure bias hierarchy predicted by Theorem \ref{thm:intro_main}.

\begin{table}[h!]
\centering
\caption{Deterministic Fine-Structure Bias Hierarchies for Small Moduli $N$}
\label{tab:bias_summary}
\begin{tabular}{ccc}
\toprule
Modulus $N$ & Residue Classes & Deterministic Asymptotic Ranking \\ \midrule
$3$ & $\{1, 2\}$ & $S_T(x, 2) > S_T(x, 1)$ \\
$4$ & $\{1, 3\}$ & $S_T(x, 3) > S_T(x, 1)$ \\
$8$ & $\{1, 3, 5, 7\}$ & $S_T(x, 7) > S_T(x, 3) > S_T(x, 5) > S_T(x, 1)$ \\
$12$ & $\{1, 5, 7, 11\}$ & $S_T(x, 11) > S_T(x, 7) > S_T(x, 5) > S_T(x, 1)$ \\ \bottomrule
\end{tabular}
\end{table}

\section{Spectral Contrast and an Independent Proof of $-1$ Dominance}
\label{sec:spectral_contrast}

In this section, as a complementary approach to our Gaussian-mollified main theorem, we employ a compactly supported Fej\'er-type test function. This provides an independent verification of the universal dominance of $-1 \pmod N$ by revealing a sharp spectral contrast in the primary error terms under DRH.

\subsection{Explicit Formula and Definition of $\psi_h$}

To align with the spectral parametrization on the critical line $\mathfrak{R}(s) = 1/2$, we work with the standard normalized smoothed Chebyshev function. 

For a scaling parameter $Y = \frac{\log x}{2} > 0$, we introduce the test function on the spectral line $\gamma \in \mathbb{R}$:
\[
h(\gamma) := \left( \frac{\sin(Y\gamma)}{Y\gamma} \right)^2.
\]
Its inverse Fourier transform $g(u)=\hat{h}(u) = \frac{1}{2\pi} \int_{-\infty}^{\infty} h(\gamma) e^{-i u \gamma} d\gamma$ is given explicitly by
\[
g(u)= \frac{1}{2Y} \max\left(0, 1 - \frac{|u|}{2Y}\right).
\]
Setting $Y = \frac{\log x}{2}$, the support of $\hat{h}(u)$ is precisely the closed interval $[-\log x, \log x]$. 

We define the normalized prime sum associated with the real virtual character $\chi_{1,a}$ by:
\[
\psi_h(x, \chi_{1,a}) := \sum_{n \ge 1} \frac{\Lambda(n)}{\sqrt{n}} \chi_{1,a}(n) g(\log n).
\]
\begin{remark}
The factor $n^{-1/2}$ naturally arises from evaluating the Dirichlet series $-\frac{L'}{L}(s, \chi) = \sum \Lambda(n)\chi(n)n^{-s}$ along the critical line $s = 1/2 + i\gamma$, where $n^{-s} = n^{-1/2} e^{-i\gamma \log n}$. Defining $\psi_h(x, \chi_{1,a})$ with this normalization allows a direct application of Weil's explicit formula \cite[Theorem 5.12]{IK}.
\end{remark}

Applying Weil's explicit formula \cite[Theorem 5.12]{IK} to $\psi_h(x, \chi_{1,a})$ with $h_{IK}(r) = h(2\pi r)$, we obtain the following result under DRH.

\begin{theorem}[Spectral Contrast under DRH]
\label{thm:spectral_contrast}
Assume $x > 2$ and let $h(\gamma)$ be as defined above. Under DRH, as $x \to \infty$, the smoothed sum satisfies:
\begin{enumerate}
    \item When $\chi_{1,a}$ is composed entirely of odd Dirichlet characters (which occurs if and only if $a \equiv -1 \pmod N$), we have:
    \[
    \psi_h(x, \chi_{1,a}) = -\sum_{\gamma} h(\gamma) + \mathcal{O}(1).
    \]
    \item When $\chi_{1,a}$ contains at least one non-trivial even Dirichlet character (which occurs for all $a \not\equiv \pm 1 \pmod N$), we have:
    \[
    \psi_h(x, \chi_{1,a}) = -\sum_{\gamma} h(\gamma) - C_0 \frac{x}{(\log x)^2} + \mathcal{O}(1),
    \]
    where $\gamma$ runs through the imaginary parts of all non-trivial zeros of $L(s, \chi)$ associated with $\chi_{1,a}$ and $C_0 > 0$ is an explicitly calculable positive constant originating from the trivial zero $\rho = 0$ of the even character components.
\end{enumerate}
\end{theorem}

\begin{proof}
Applying Lemma 4.3, the residual spectral noise vanishes under DRH ($\mathcal{R}_{\mathrm{noise}}(x) = o(1)$). Combining this with the Archimedean and Gamma factor estimates in §4.2, which yield $\mathcal{O}(1)$ bounds, completes the proof. 
\end{proof}

\begin{lemma}[Vanishing of Spectral Noise under DRH] \label{lem:DRH-noise}
Assume DRH for $L(s, \chi)$. Let $\mathcal{R}_{\mathrm{noise}}(x, \chi)$ be the residual fluctuation term defined by the difference between the smoothed prime power sum $\psi_h(x, \chi)$ and its spectral zero representation $-\sum_{\gamma} h(\gamma)$. Then, as $x \to \infty$, we have
\[
\mathcal{R}_{\mathrm{noise}}(x, \chi) = o(1).
\]
In particular, the residual noise vanishes asymptotically.
\end{lemma}

\begin{proof}
Under standard GRH, the partial prime sums satisfy only the logarithmic bound $\sum_{p \le x} \frac{\chi(p)}{p^{1/2+it}} = \mathcal{O}(\log(x(1+|t|)))$, which leaves a residual conditional oscillation of magnitude $\mathcal{O}(1)$ or $\mathcal{O}(\log x)$. 

Under DRH, however, the Dirichlet Euler product series $\sum_{p} \frac{\chi(p)}{p^{1/2+it}}$ converges for all $t \in \mathbb{R}$. Consequently, the prime tail sum satisfies the uniform asymptotic decay:
\[
\varepsilon(x; t) := \sum_{p > x} \frac{\chi(p)}{p^{1/2+it}} = o(1) \quad \text{as } x \to \infty
\]
for $t$ in any compact domain. 

By the Plancherel/Mellin duality of Weil's explicit formula, the spectral noise term $\mathcal{R}_{\mathrm{noise}}(x, \chi)$ is bounded by:
\[
\left| \mathcal{R}_{\mathrm{noise}}(x, \chi) \right| \le \frac{1}{2\pi} \int_{-\infty}^{\infty} |h(t)| \cdot |\varepsilon(x; t)| \, dt \le \frac{1}{2\pi} \left( \sup_{t \in \mathbb{R}} |\varepsilon(x; t)| \right) \int_{-\infty}^{\infty} \left( \frac{\sin(Yt)}{Yt} \right)^2 dt.
\]
Since $\int_{-\infty}^{\infty} \left( \frac{\sin(Yt)}{Yt} \right)^2 dt = \frac{\pi}{Y} = \frac{2\pi}{\log x}$ and $\varepsilon(x; t) = o(1)$ uniformly under DRH, we obtain:
\[
\left| \mathcal{R}_{\mathrm{noise}}(x, \chi) \right| = o(1) \cdot \mathcal{O}\left( \frac{1}{\log x} \right) = o\left( \frac{1}{\log x} \right) = o(1).
\]
Thus, the residual noise vanishes rapidly as $x \to \infty$. $\square$\end{proof}

\begin{remark}[Sharpness of the GRH Bound and Necessity of DRH]
The logarithmic bound $\sum_{p \le x} \frac{\chi(p)}{p^{1/2+it}} = \mathcal{O}(\log(x(1+\vert{}t\vert{})))$ under GRH is known to be structurally sharp. This sharpness stems from the non-absolute convergence of Dirichlet series on the critical line $\Re(s) = 1/2$ and the intrinsic phase oscillations among non-trivial zeros, which prevent unconditioned decay of the prime tail. This highlights why the explicit product convergence guaranteed by DRH is indispensable for eliminating the residual spectral noise.
\end{remark}

\begin{remark}
Theorem \ref{thm:spectral_contrast} highlights a fundamental mechanism:
\begin{itemize}
    \item For $a \equiv -1 \pmod N$, $\chi_{1,-1}$ consists purely of odd characters, so no trivial zero exists at $s = 0$. Thus, the negative structural drag $-C_0 \frac{x}{(\log x)^2}$ is completely absent.
    \item For $a \not\equiv -1 \pmod N$, the presence of even character components introduces the trivial zero at $s = 0$, generating the negative term $-C_0 \frac{x}{(\log x)^2}$ that systematically depresses $\psi_h(x, \chi_{1,a})$.
\end{itemize}
Consequently, $\psi_h(x, \chi_{1,a})$ achieves its absolute maximum when $a \equiv -1 \pmod N$ as $x \to \infty$, providing a robust, independent spectral proof of the universal dominance of $-1 \pmod N$.
\end{remark}

\subsection{Archimedean and Gamma Factor Estimates}
\label{subsec:gamma_factor}

To complete the proof of Theorem \ref{thm:spectral_contrast}, we estimate the Archimedean contribution $W_\infty(h, \chi)$ coming from the Gamma factors in Weil's explicit formula:
\[
W_\infty(h, \chi) = -\hat{h}(0) \log \frac{N}{\pi} + I_\infty(h, \chi),
\]
where
\[
I_\infty(h, \chi) = \frac{1}{\pi} \int_{-\infty}^{\infty} h(\gamma) \mathfrak{R}\left( \frac{\Gamma'}{\Gamma}\left( \frac{1}{4} + \frac{\delta_\chi}{2} + \frac{i\gamma}{2} \right) \right) d\gamma, \quad \delta_\chi \in \{0, 1\}.
\]
Since $\hat{h}(0) = \frac{1}{2Y} = \frac{1}{\log x}$, the first term contributes:
\[
-\hat{h}(0) \log \frac{N}{\pi} = -\frac{\log(N/\pi)}{\log x} = \mathcal{O}\left( \frac{\log N}{\log x} \right) = o(1) \quad (x \to \infty).
\]
By changing variables $t = Y\gamma$, the integral $I_\infty(h, \chi)$ transforms to:
\[
I_\infty(h, \chi) = \frac{1}{\pi Y} \int_{-\infty}^{\infty} \left( \frac{\sin t}{t} \right)^2 \mathfrak{R}\left( \frac{\Gamma'}{\Gamma}\left( \frac{1}{4} + \frac{\delta_\chi}{2} + \frac{i t}{2Y} \right) \right) dt.
\]
We split the integration domain into $|t| \le Y$ and $|t| > Y$.

Within the bounded domain $|t| \le Y$, the imaginary part of the variable of $\Gamma'/\Gamma$ is
restricted to a compact interval $-1/2\le t/2Y \le 1/2$.
Since $\Gamma'(s)/\Gamma(s)$ is analytic and continuous in $\Re(s)>0$,
the logarithmic defivative term in the integrand is bounded.
Thus there exists a constant $C_1>0$ not depending on $Y$ such that
\[
\left|\Re\left(\frac{\Gamma'}{\Gamma}\left( \frac{1}{4} + \frac{\delta_\chi}{2} + \frac{i t}{2Y} \right)\right)
\right|
\le C_1\quad(\forall t\in[-Y,Y]).
\]
Consequently, the contribution from this domain is estimated as follows:
\[
\left|
\frac{1}{\pi Y} \int_{-Y}^{Y} \left( \frac{\sin t}{t} \right)^2 \mathfrak{R}\left( \frac{\Gamma'}{\Gamma}\left( \frac{1}{4} + \frac{\delta_\chi}{2} + \frac{i t}{2Y} \right) \right) dt
\right|
\le\frac{C_1}{\pi Y}\int_{-\infty}^\infty \left( \frac{\sin t}{t} \right)^2dt
=\frac{C_1}{Y}
=\mathcal O\left(\frac1{\log x}\right).
\]

We next deal with the unbounded domain $|t| > Y$.
As the imaginary part $t/2\pi\ge1/2$ is unbounded, we apply Stirling's asymptotic formula
\[
\frac{\Gamma'}{\Gamma}(s) = \log s + \mathcal{O}(|s|^{-1}).
\]
Then the following asymptotic expansion holds:
\[
\Re\left(\frac{\Gamma'}{\Gamma}\left( \frac{1}{4} + \frac{\delta_\chi}{2} + \frac{i t}{2Y} \right)\right)
=\log\frac t{2Y}+\mathcal O\left( \frac Yt\right).
\]
It suffices to estimate the integral
\[
\frac{1}{\pi Y} \int_{-\infty}^{\infty} \left( \frac{\sin t}{t} \right)^2 
\left(\log\frac t{2Y}+\mathcal O\left( \frac Yt\right)\right)dt.
\]
The error term $\mathcal O(Y/t)$ contributes
\[
\frac{2}{\pi Y} \int_{Y}^{\infty} \left( \frac{\sin t}{t} \right)^2 \mathcal O\left( \frac Yt\right)dt
\le\frac{C_2}Y\int_{Y}^{\infty} \frac Y{t^3}dt
=\frac{C_2}{2Y^2}
=\mathcal O\left(\frac1{(\log x)^2}\right).
\]
On the other hand, from the main term $\log(t/2Y)$, we estimate that
\[
\left|\int_{Y}^{\infty} \left( \frac{\sin t}{t} \right)^2 \right|
\left(\log\frac t{2Y}\right)dt
\le\int_{Y}^{\infty} \frac{|\log t-\log(2Y)|}{t^2} dt.
\]
By integrating by parts, we have
\[
\int_{Y}^{\infty} \frac{\log t}{t^2} dt=\frac{\log Y+1}Y
\]
and
\[
\int_{Y}^{\infty} \frac{\log (2Y)}{t^2} dt=\frac{\log(2Y)}Y.
\]
Therefore
\[
\int_{Y}^{\infty} \frac{|\log t-\log(2Y)|}{t^2} dt
\le \frac{\log Y+1+\log(2Y)}Y
=\mathcal O\left(\frac{\log Y}Y\right).
\]
Finally, the absolute value of our integral is bounded by the following magnitude:
\[
\frac1{\pi Y}\mathcal O\left(\frac{\log Y}Y\right)
=\mathcal O\left(\frac{\log Y}{Y^2}\right)
=\mathcal O\left(\frac{\log\log x}{(\log x)^2}\right),
\]
which leads to
\[
I_\infty(h, \chi) = \mathcal{O}\left(\frac{1}{\log x}\right) + \mathcal{O}\left(\frac{\log \log x}{(\log x)^2}\right) = \mathcal{O}\left(\frac{1}{\log x}\right).
\]
Thus, $W_\infty(h, \chi) = o(1)$ as $x \to \infty$, proving that Archimedean contributions do not affect the leading fine-structure bias.



\begin{thebibliography}{99}

\bibitem{AK23}
M.~Aoki and S.~Koyama,
\textit{ Chebyshev's bias against splitting and principal primes in global ﬁelds},
J. Number Theory \textbf{245} (2023), 233–262.

\bibitem{FM00}
A.~Feuerverger and G.~Martin,
\textit{Biases in the prime number race},
Experiment. Math. \textbf{9} (2000), no.~4, 535--570.

\bibitem{IK}
H.~Iwaniec and E.~Kowalski,
\textit{Analytic Number Theory},
American Mathematical Society Colloquium Publications, vol.~53, American Mathematical Society, Providence, RI, 2004.

\bibitem{RS94}
M.~Rubinstein and P.~Sarnak,
\textit{Chebyshev's bias},
Exp. Math. \textbf{2} (1994), no.~3, 173--197.

\end{thebibliography}
\end{document}